\documentclass{article}

\usepackage{amsmath,amsfonts,amsthm,amssymb,graphicx,tikz}
\usepackage[all,2cell,ps]{xy}

\theoremstyle{plain}
\newtheorem{thm}{Theorem}[section]

\theoremstyle{definition}

\newtheorem{rem}[thm]{Remark}

\title{Geodesic curves on Shimura surfaces}

\author{Ted Chinburg \and Matthew Stover}

\date{\today}

\begin{document}

\maketitle

\begin{center}
Dedicated to the memory of Colin Maclachlan
\end{center}

\section{Introduction}

A Shimura surface is the quotient of either the product $\mathbf{H}^2 \times \mathbf{H}^2$ of two hyperbolic planes or the unit ball $\mathbf{H}_{\mathbb{C}}^2$ in $\mathbb{C}^2$ by an irreducible arithmetic lattice. Examples include the normal quasiprojective varieties associated with the Hilbert and Picard modular groups, along with the solutions to many moduli problems for principally polarized abelian varieties. Special amongst the immersed projective algebraic curves on these surfaces are those which are geodesic for the metric descending from the universal covering. In this paper, we completely classify the geodesic curves on Shimura surfaces up to commensurability. A consequence of this classification is the following.

\begin{thm}\label{thm:Concise}
Let $S$ be a Shimura surface. If $S$ contains one geodesic curve, then it contains infinitely many that are pairwise incommensurable.
\end{thm}

More specifically, we give a parametrization of the commensurability classes of geodesic curves that appear on a given Shimura surface. We note that there are indeed situations, for both $\mathbf{H}^2 \times \mathbf{H}^2$ and $\mathbf{H}_{\mathbb{C}}^2$, where there are no such curves. See Theorems \ref{thm:h2xh2} and \ref{thm:Hc2fuchsian} for the precise statements. We also note that this problem is equivalent to parametrizing the commensurability classes of what are often called $\mathbb{C}$-Fuchsian subgroups of $\pi_1(S)$.

These results originally appeared in a 2012 preprint \cite{CS2012}, where we applied Theorem \ref{thm:Concise} to prove a result on generating fundamental groups of Shimura surfaces using supports of divisors made up of geodesic curves. That application can be strengthened using very recent work of M\"oller--Toledo \cite{Moller--Toledo} or Koziarz--Maubon \cite{Koziarz--Maubon}. Since the classification of geodesic curves is of independent interest, we decided to write the present note.

We now make some historical remarks. Such parameterizations are known for other classes of arithmetic lattices. The first result of this kind was the parameterization of Fuchsian subgroups of arithmetic Kleinian groups by Maclachlan and Reid \cite{MR2}. Meyer solved the problem of parametrizing commensurability classes of geodesic hyperbolic submanifolds of certain arithmetic hyperbolic $n$-manifolds \cite{Meyer}. Very recently, Parkkonen and Paulin gave a classification for Picard modular groups \cite{Parkkonen--Paulin}. This is a special case of our work, but Parkkonen--Paulin state the classification in slightly different language; we describe the equivalence between the two, which follows from elementary quaternion algebra arithmetic, at the end of \S \ref{sec:Hc2}.

The paper is organized as follows. In \S \ref{sec:generalities}, we give some general facts on lattices in Lie groups and apply them to the cases of interest in this paper. In \S \ref{sec:HxH} we give the classification for $\mathbf{H}^2 \times \mathbf{H}^2$. In \S \ref{sec:Hc2}, we give the classification for $\mathbf{H}_{\mathbb{C}}^2$.

\section{Generalities on lattices}\label{sec:generalities}

Let $G$ be a semisimple Lie group with finite center and associated symmetric space $X$ of noncompact type. Then $X$ is called \emph{Hermitian symmetric} when it carries a $G$-invariant complex structure \cite[Ch.\ VIII]{Helgason}. The quotient of $X$ by any cocompact (resp.\ nonuniform) lattice $\Gamma$ in $G$ is a projective (resp.\ quasiprojective) algebraic variety \cite{Baily}. Suppose $H$ is a semisimple Lie subgroup of $G$ with finite center and associated symmetric space $Y$. There is then a totally geodesic embedding $Y \to X$, and the totally geodesic subspaces of $\Gamma \backslash X$ arise from the subgroups $H$ of this kind such that $\Gamma \cap H$ is a lattice in $H$; see \cite[\S1.4]{Vinberg}.

When $\Gamma$ is an \emph{arithmetic} subgroup of $G$ (see \cite[\S3.6]{Vinberg} for definitions), it follows that $\Gamma \cap H$ is an arithmetic subgroup of $H$. More precisely, suppose $K_0$ is a number field and that $\mathcal G$ is a $K_0$-algebraic group such that $\mathcal G(\mathbb{R} \otimes_{\mathbb{Q}} K_0)$ modulo compact factors is isomorphic to $G$. Let $\Gamma$ be an arithmetic lattice commensurable with the lattice of $K_0$-integral points of $\mathcal G$. Then connected totally geodesic submanifolds of $\Gamma \backslash X$ are determined by semisimple $F$-algebraic subgroups of $\mathcal G$ for certain $F \subseteq K_0$.

The cases of interest here are where $G$ is either $\mathrm{SL}_2(\mathbb{R}) \times \mathrm{SL}_2(\mathbb{R})$ or $\mathrm{SU}(2, 1)$. The respective hermitian symmetric domains are the product $\mathbf{H}^2 \times \mathbf{H}^2$ of two hyperbolic planes or the complex hyperbolic plane $\mathbf{H}_{\mathbb{C}}^2$. Therefore, the possible holomorphically embedded totally geodesic submanifolds are \emph{Fuchsian curves}, that is, quotients of $\mathbf{H}^2$ by cocompact arithmetic subgroups of $\mathrm{SL}_2(\mathbb{R})$. These lattices are often called \emph{arithmetic Fuchsian groups}. See \cite{Vigneras} and \cite{Maclachlan--Reid} for an account of the basic theory of arithmetic Fuchsian groups.

We recall the construction of lattices in $\mathrm{SL}_2(\mathbb{R})$ and $\mathrm{SL}_2(\mathbb{R}) \times \mathrm{SL}_2(\mathbb{R})$ in \S \ref{sec:HxH} and consider $\mathrm{SU}(2, 1)$ in \S\ref{sec:Hc2}.

\section{Lattices acting on $\mathbf H^2$ and $\mathbf H^2 \times \mathbf H^2$}\label{sec:HxH}

We recall from \cite[\S3]{Borel2} the construction of irreducible arithmetic lattices acting on products of hyperbolic planes. Let $\mathbf{H}^2$ be the hyperbolic plane and $G$ (resp.\ $X$) be the product of $n$ copies of $\mathrm{SL}_2(\mathbb{R})$ (resp.\ $\mathbf H^2$). Suppose that $K$ is a totally real number field and $A$ a quaternion algebra over $K$ that is ramified at exactly $n$ archimedean places of $K$. Then there is an algebra surjection $\pi:\mathbb{R} \otimes_K A \to \mathrm{M}_2(\mathbb{R})^n$ whose kernel is a product of $[K : \mathbb{Q}] - n$ copies of the Hamilton quaternion algebra over $\mathbb{R}$.
 
Let $\mathcal{O}$ be an order of $A$ and let $\mathcal{O}^1$ be the multiplicative subgroup of elements with reduced norm one in $\mathcal{O}$. Then $\pi(\mathcal{O}^1)$ is an irreducible lattice in $\mathrm{SL}_2(\mathbb{R})^n$. An irreducible lattice $\Gamma < \mathrm{SL}_2(\mathbb{R})^n$ is arithmetic when it is commensurable with a group $\pi(\mathcal{O}^1)$ constructed in the above way. Two arithmetic lattices $\Gamma$ are commensurable if and only if they have the same associated totally real field $k$ and quaternion algebra, modulo the action of $\mathrm{Aut}(k / \mathbb{Q})$ \cite[Thm.\ 8.4.7]{Maclachlan--Reid}. An arithmetic lattice $\Gamma$ is cocompact if and only if $A \not\cong \mathrm{M}_2(K)$.

When $n =1$ such $\Gamma$ are arithmetic Fuchsian groups. When $n = 2$, a lattice $\Gamma < G$ is irreducible if its projection onto any factor is dense in the analytic topology. A lattice fails to be irreducible in this case if and only if there is a finite index subgroup that is the direct product of two Fuchsian groups. All irreducible lattices in $G$ are arithmetic by Margulis's Arithmeticity Theorem \cite[p.\ 2]{Margulis}.

\begin{thm}\label{thm:h2xh2}
Let $\Gamma \subset \mathrm{SL}_2(\mathbb{R}) \times \mathrm{SL}_2(\mathbb{R})$ be an irreducible lattice defined via the quaternion algebra $A$ over the totally real number field $K$. Then, there is a one-to-one correspondence between
\begin{enumerate}

\item commensurability classes of arithmetic Fuchsian subgroups of $\Gamma$ and

\item degree two subfields $K_0 \subset K$ and $\mathrm{Aut}(K_0 / \mathbb{Q})$-isomorphism classes of $K_0$-subalgebras $B \subset A$ such that $A = K \otimes_{K_0} B$.

\end{enumerate}
In particular, $\Gamma$ contains infinitely many commensurability classes of arithmetic Fuchsian groups if and only if it contains one arithmetic Fuchsian subgroup.
\end{thm}
\begin{proof}
We analyze the possible simple subalgebras of $A$. We first argue that a subalgebra as in \emph{2}.\ always produces an arithmetic Fuchsian subgroup of $\Gamma$. This is equivalent to showing that every such subalgebra splits at exactly one real place of $K_0$. This follows immediately from analyzing the possibly splitting behavior of quaternion algebras under tensor products. Indeed, at real place of $K_0$ at which $B$ is isomorphic to Hamilton's quaternions (resp.\ $\mathrm{M}_2(\mathbb{R})$) cannot extend to a real place of $K$ at which $A$ is isomorphic to $\mathrm{M}_2(\mathbb{R})$ (resp.\ Hamilton's quaternions), since Hamilton's quaternions do not embed in $\mathrm{M}_2(\mathbb{R})$ and vice versa.

Conversely, let $\Sigma$ be an arithmetic Fuchsian group with associated totally real field $K_0$ and $K_0$-quaternion algebra $B$, and suppose that $\Sigma$ is a Fuchsian subgroup of $\Gamma$. The inclusion of $\Sigma$ into $\Gamma$ induces an embedding of algebras $B \to A$. Indeed, consider
\[
\Sigma \subset \Gamma \subset A,
\]
where we can assume that $\Gamma \subset A$ possibly after passing to a subgroup of finite index, which has no effect on the commensurability classes of arithmetic Fuchsian subgroups. If $K_0$ is the invariant trace field of $\Sigma$ (see \cite{Maclachlan--Reid, MR2}), then the $K_0$-subalgebra of $A$ generated by $\Sigma$ is $K_0$-isomorphic to $B$. It follows immediately that $A = K \otimes_{K_0} B$, so we must show that $K_0$ is a quadratic subfield of $K$. However, $B$ splits at exactly one infinite place of $K_0$, since it is the quaternion algebra associated with an arithmetic Fuchsian group. Then the number of real places of $K$ at which $K \otimes_{K_0} B$ splits is $[K : K_0]$, which is $2$ by assumption on $\Gamma$, so $K_0$ is a quadratic subfield of $K$. Therefore, Fuchsian subgroups of $\Gamma$ arise from subalgebras as in \emph{2}.

The last statement of the theorem can be see as follows. Let $B'$ be another quaternion algebra over $K_0$ such that the places of $K_0$ that ramify in exactly one of $B$ or $B'$ do not split in $K$. Then $K \otimes_{K_0} B \cong K \otimes_{K_0} B'$, so $B'$ defines another arithmetic Fuchsian subgroup of $\Gamma$. Taking $B'$ not $\mathrm{Aut}(K_0 / \mathbb{Q})$-conjugate to $B$ implies that the associated arithmetic Fuchsian subgroups form a new commensurability class in $\Gamma$. The Cebotarev density theorem implies that there are infinitely many primes of $K_0$ that do not split in $K$, hence there are infinitely many distinct choices of $B'$ that produce distinct commensurability classes.
\end{proof}

\begin{rem}
The easiest way to build arithmetic subgroups of $\mathrm{SL}_2(\mathbb{R}) \times \mathrm{SL}_2(\mathbb{R})$ containing no arithmetic Fuchsian subgroups is to take $K$ a totally real field with odd degree over $\mathbb{Q}$ and take a $K$-quaternion algebra unramified at exactly two real places. Then $K$ contains no quadratic subfields, hence no arithmetic Fuchsian subgroups.
\end{rem}

\section{Lattices acting on $\mathbf{H}_{\mathbb{C}}^2$}\label{sec:Hc2}

First, we quickly recall the standard construction of $\mathbf{H}_{\mathbb{C}}^2$. See \cite[Ch.\ 3]{Goldman} for further details. Let $V$ be a rank $3$ vector space over $\mathbb{C}$ and $h$ a hermitian form on $V$ of signature $(2, 1)$. If $V_-$ denotes the subspace of $h$-negative vectors, then $\mathbf{H}_{\mathbb{C}}^2$ is the space $\mathbb{P}(V_-)$ of $h$-negative lines. There is a natural biholomorphism from $\mathbf{H}_{\mathbb{C}}^2$ with the metric determined by $h$ to the unit ball in $\mathbb{C}^2$ with the Bergman metric.

Totally geodesic holomorphic embeddings of $\mathbf{H}^2$ in $\mathbf{H}_{\mathbb{C}}^2$ arise from rank $2$ subspaces of $V$ on which the restriction of $h$ to $V$ has signature $(1, 1)$. Taking the $h$-orthogonal complement, it follows that holomorphic totally geodesic embeddings of $\mathbf{H}^2$ into $\mathbf{H}_{\mathbb{C}}^2$ are in one-to-one correspondence with $h$-positive lines in $V$. See \cite[\S3.3.1]{Goldman} and \cite[\S12.2]{Bergeron--Clozel}.

There are two constructions of arithmetic subgroups of $\mathrm{SU}(2, 1)$. By \cite[Ch.\ 8]{Bergeron--Clozel} or \cite{Moller--Toledo}, the only arithmetic lattices in $\mathrm{SU}(2, 1)$ that contain totally geodesic surfaces are those of so-called \emph{simple type} (also sometimes called \emph{first type}). These are defined as follows.

Recall that a \emph{CM-pair} $K / K_0$ is a totally imaginary quadratic extension of a totally real number field. Let $z \mapsto \overline z$ be the nontrivial automorphism of $K$ over $K_0$ and $h$ be a hermitian form on $K^3$ such that $h$ is indefinite at precisely one $\mathrm{Gal}(K / K_0)$-conjugate pair of embeddings of $K$ into $\mathbb{C}$. If $\mathcal{O}$ is the integer ring of $K$, the $K_0$-algebraic group
\[
\mathrm{SU}(h) = \{ x \in \mathrm{SL}_3(K)\ :\ {}^t \overline x h x = h \},
\]
contains the discrete subgroup $\Gamma_{\mathcal{O}}^1 = \mathrm{SU}(h) \cap \mathrm{SL}_r(\mathcal{O})$. Then $\Gamma_{\mathcal{O}}^1$ projects to a lattice in $\mathrm{SU}(2, 1)$, since $h$ is indefinite above exactly one real embedding of $K_0$. It is known that the commensurability class of $\Gamma_{\mathcal{O}}^1$ in $\mathrm{SU}(2, 1)$ depends only on $K$, not the choice of hermitian form (see \cite[\S1.2]{Prasad--Yeung}). The lattice $\Gamma_{\mathcal{O}}^1$ is cocompact if and only if $K_0 \neq \mathbb{Q}$.

Recall that the arithmetic quotients of $\mathbf{H}_{\mathbb{C}}^2$ that are not of simple type contain no totally geodesic curves. The following classifies and parametrizes the possible arithmetic Fuchsian subgroups of arithmetic lattices in $\mathrm{SU}(2, 1)$ of simple type.

\begin{thm}\label{thm:Hc2fuchsian}
Let $\Gamma < \mathrm{SU}(2, 1)$ be an arithmetic lattice of simple type with associated CM-pair $K / K_0$. Then, there is a one-to-one correspondence between
\begin{enumerate}

\item commensurability classes of arithmetic Fuchsian subgroups of $\Gamma$ corresponding to totally geodesic projective algebraic curves on $\Gamma \backslash \mathbf H_{\mathbb{C}}^2$ and

\item $\mathrm{Aut}(K_0 / \mathbb{Q})$-isomorphism classes of $K_0$-quaternion algebras $A$ ramified at all but one infinite place of $K_0$ and at any finite set of nonarchimedean places that do not split in $K / K_0$.

\end{enumerate}
In particular, any such $\Gamma$ contains infinitely many distinct commensurability classes of such arithmetic Fuchsian subgroups.
\end{thm}
\begin{proof}
Assume that $h$ is a hermitian form on $V = K^3$ that is indefinite above exactly one archimedean place of $K_0$. Let $W \cong K^2 \subset V$ be the $h$ orthogonal complement of a line in $V$ that is $h$-positive at the archimedean place of $K_0$ over which $h$ is indefinite, and let $h_W$ be the restriction of $h$ to $W$. By our earlier discussion of holomorphic embeddings of $\mathbf{H}^2$ into $\mathbf{H}_{\mathbb{C}}^2$, the pair $(W,h_W)$ determines a maximal arithmetic Fuchsian subgroup $\Sigma_W$ of $\Gamma$ associated with a holomorphic curve on $\Gamma \backslash \mathbf{H}_{\mathbb{C}}^2$, and all such curves arise in this way. We now proceed to determine the $K_0$-quaternion algebra associated to this arithmetic Fuchsian group $\Sigma_W$.

There is a natural homomorphism from $\Sigma_W$ to $\mathrm{SL}_2(K)$. Furthermore, the $K_0$-subalgebra $B$ of $\mathrm{M}_2(K)$ generated by $\Sigma_W$ is four-dimensional and noncommutative, so it is a quaternion algebra. Since $h_W$ is indefinite at exactly one infinite place of $K_0$ and $B$ embeds into $\mathrm{M}_2(K)$ (i.e., $K$ splits $B$), $B$ must satisfy condition \emph{2}.\ of the theorem.

Conversely, suppose $B$ is a quaternion algebra over $K_0$ satisfying the conditions in \emph{2}.\ We must show that there is a commensurability class of arithmetic Fuchsian subgroups of $\Gamma$ associated with $B$. Condition \emph{2}.\ implies that we have an embedding of $K$ into $B$, so $B$ is a two-dimensional left vector space over $K$. The nontrivial automorphism of $K$ over $K_0$ extends to an anti-involution $z \to \overline{z}$ of $B$ over $K_0$.

There is a left $K$-linear action of $B$ on itself under which $z \in B$ acts on $x \in B$ by $x \mapsto x \cdot \overline{z}$. This embeds $B$ into $\mathrm{Mat}_2(K)$. Let $\mathrm{Tr}:B \to K$ be the composition of this embedding with the trace. Identifying $B$ with $K^2$, we obtain a hermitian form on $K^2$ by
\[
h_B(x, y) = \mathrm{Tr}(x \overline{y}),
\]
for $x, y \in B$.

Now, consider the hermitian form $H_B$ on $K^3$ defined as follows. Write $K^3 = W \oplus \ell$, where $W$ is two-dimensional and $\ell$ is a line. Let the restriction of $H_B$ to $W$ be $h_B$, let $\ell$ be $H_B$-orthogonal to $W$, and let the $H_B$-norm of some generator of $\ell$ be the totally positive number $\det(h) / \det(h_B)$ in $K_0$. By \cite[Ch.\ 10, Ex.\ 1.8(iii)]{Scharlau}, $K^3$ equipped with $H_B$ is isomorphic as a hermitian vector space to $K^3$ equipped with $h$. This realizes the elements $B^1$ of with with reduced norm $1$ inside $\mathrm{SU}(h)$ as the stabilizer of the subspace $W$. Then, $\Gamma \cap B$ is an arithmetic Fuchsian subgroup of $\Gamma$ with associated quaternion algebra $B$. This completes the proof.
\end{proof}

\begin{rem}
In \cite{Parkkonen--Paulin}, Parkkonen and Paulin recently parametrized the $\mathbb{C}$-Fuchsian subgroups of the Picard modular groups. Recall that the Picard modular groups define the commensurability classes of nonuniform arithmetic lattices in $\mathrm{PU}(2, 1)$, and are defined by hermitian forms over imaginary quadratic fields. In other words, the CM pair is $K / \mathbb{Q}$ for some imaginary quadratic field $K$.

Theorem \ref{thm:Hc2fuchsian} identifies Fuchsian subgroups of a Picard modular groups with $\mathbb{Q}$-quaternion algebras $A$ that are unramified at the infinite place of $\mathbb{Q}$ and that ramify at a (possibly empty) set of rational primes that do not split in $K$. In particular, $K$ embeds as a subfield of $A$. Standard arguments show that $A$ then has a Hilbert symbol of the form
\[
A = \begin{pmatrix} \frac{n, D_K}{\mathbb{Q}} \end{pmatrix},
\]
where $n$ is a positive square-free integer and $D_K$ is the discriminant of $K$, which is the classification given in \cite{Parkkonen--Paulin}.
\end{rem}

\subsection*{Acknowledgments}
Chinburg was supported by the National Science Foundation under Grant Numbers DMS 1360767 and DMS 1265290. Stover was supported by the National Science Foundation under Grant Number NSF 1361000. Stover also acknowledges support from U.S. National Science Foundation grants DMS 1107452, 1107263, 1107367 ``RNMS: GEometric structures And Representation varieties'' (the GEAR Network).

\bibliography{GeodesicCurves}

\end{document}